\documentclass[11pt]{article}
\usepackage{amsmath}
\usepackage{amsfonts}
\usepackage{amsthm}
\usepackage{amssymb, setspace}
\usepackage{color,graphicx, epsfig, geometry, hyperref,fancyhdr}
\usepackage{float}
\newtheorem{theorem}{Theorem}[section]

\newtheorem{corollary}{Corollary}[section]
\newtheorem{lemma}{Lemma}[section]
\newtheorem{definition}{Definition}[section]
\newtheorem{assumption}{Assumption}[section]

\newcommand{\N}{\mathbb{N}}

\newcommand{\R}{\mathbb{R}}
\newcommand{\C}{\mathbb{C}}

\newcommand{\eps}{\varepsilon}

\newcommand{\ep}{\varepsilon}

\graphicspath{{paper-pics/}}
\usepackage{graphicx}
\usepackage{epstopdf}
\DeclareGraphicsExtensions{.eps,.jpg}

\begin{document}
\setlength{\parskip}{1mm}
\setlength{\oddsidemargin}{0.1in}
\setlength{\evensidemargin}{0.1in}
\lhead{}
\rhead{}

\begin{center}
{\bf \Large \noindent Near field imaging of small isotropic and extended anisotropic scatterers }\\
\vspace{0.25in}
Isaac Harris\footnote{Department of Mathematics, Texas A $\&$ M University, College Station, Texas 77843-3368, (Corresponding Author) E-mail: iharris@math.tamu.edu} 
and Scott Rome\footnote{Email: romescott@gmail.com } 
\end{center}

\begin{abstract}
\noindent In this paper, we consider two time-harmonic inverse scattering problems of reconstructing penetrable inhomogeneous obstacles from near field measurements. First we appeal to the Born approximation for reconstructing small isotropic scatterers via the MUSIC algorithm. Some numerical reconstructions using the MUSIC algorithm are provided for reconstructing the scatterer and piecewise constant refractive index using a Bayesian method. We then consider the reconstruction of an anisotropic extended scatterer by {a modified linear sampling method and the factorization method applied }to the near field operator. This provides a rigorous characterization of the support of the  anisotropic obstacle in terms of a range test derived from the measured data. Under appropriate assumptions on the material parameters, the derived factorization can be used to determine the support of the medium without a-priori knowledge of the material properties. 
\end{abstract}



\section{Introduction}
Qualitative/Sampling methods have been used to solve multiple inverse problems such as parameter identification and shape reconstruction. These methods have been used to solve inverse boundary value problems for elliptic, parabolic and hyperbolic partial differential equation. See \cite{heatlsm} and \cite{wavelsm} for examples of the linear sampling methods applied to parabolic and hyperbolic equations, respectively. The vast available literature is a representative of the many directions that this research has taken (see \cite{p1}, \cite{coltonkress}, \cite{coyle}, \cite {kirschbook} and the references therein).

The factorization method is a qualitative method which develops a range test to determine the support of the unknown scatterer. In general, we have that 
$$\phi_z \in \text{Range} \big( N_{\sharp}^{1/2} \big) \iff z \, \, \text{is inside the obstacle}$$ 
where $\phi_z$ is known and $N_{\sharp}$ is a positive selfadjoint compact operator defined by the measurements. By appealing to the range test, we have that any two obstacles admitting the same data must be equal, proving uniqueness of the inverse problem. Since the support of the scatterer is connected to the range of a compact operator which is known, evaluating the indicator function derived from the factorization method amounts to applying PicardÕs criteria, which only requires the singular values and functions of a known operator. This makes the factorization method computationally cheap to implement and analytically rigorous.
{Using the characterization of the obstacle given by the factorization method we propose a modified linear sampling method as an alternative sampling algorithm for reconstructing the obstacle. The indicator function proposed for the modified linear sampling method can be shown to satisfy a similar analytically rigors characterization of the obstacle. 
}

In this paper, we consider two problems associated with inverse scattering. The first problem we consider is the reconstruction of small isotropic obstacles and the refractive index from near field data. To do so, we derive the Multiple-Signal-Classification (MUSIC) algorithm which gives an analytically rigorous method of reconstructing small obstacles. Once we have reconstructed the obstacles, we will give a few example of how one can reconstruct constant refractive indices using a Bayesian method. Next we consider the reconstructing an anisotropic obstacle from near field data via the factorization method
{and modified linear sampling method.
}
This problem has been considered in \cite{lsmaniso} and \cite{fmaniso} for far field measurements using the linear sampling and factorization methods, respectively. The factorization method for an inhomogeneous isotropic media with near field measurements was studied in \cite{nf-fm-isotropic}. We analyze the factorization method applied to the near field operator for an anisotropic obstacle to derive a simple numerical algorithm to reconstruct the scatterer. For the case of an anisotropic media, the unique determination of the support is optimal because a matrix-valued coefficient is in general not uniquely determined by the scattering data (see for e.g. \cite{uniqueness-aniso}). The analysis done in this paper is of a similar flavor to the work done in \cite{fmdefect} and \cite{fmshawn}.

\section{Recovering small {volume} isotropic scatterers and the refractive index}
\subsection{Scattering by small isotropic scatterers}
In this section, we formulate the direct time-harmonic  scattering problem in $\R^d$ for small isotropic scatterers. The scattering obstacle may be made up of multiple simply components, but is not necessarily simply connected. We are interested in using near field measurements for `small' obstacles embedded in a homogenous media. Let $u^i( \cdot \, ,y)$ be the incident field that originates from the source point $y$ located on the curves $C$. We will assume that the curve $C$ is class $\mathcal{C}^2$ and that the incident field satisfies the Helmholtz equation in $\R^d \setminus\{y\}$ for $d=2$ or $3$. Now let 
$$D= \bigcup\limits_{m=1}^{M} D_m \quad \textrm{ with } \quad D_m=\big(z_m + \ep B_m \big),$$
where $\ep >0$ is a small parameter and $B_m$ is a domain with a piece-wise smooth boundary that is centered at the origin. Outside of the scattering obstacle(s) $D_m$, the material parameters are homogeneous isotropic with refractive index scaled to one. Here we denote $n(x)$ the refractive index of the homogenous background with the (possibly inhomogeneous) obstacle $D$ in ${\mathbb R}^d$ by 
 $$n(x)= 1 \big(1-\chi_D \big)+\sum\limits_{m=1}^{M} n_m (x) \chi_{D_m} $$
 where $\chi$ is the characteristic function.
The refractive index of the obstacle(s) $D_m$ is represented by a continuous scalar function $n_m$ such that the real and imaginary parts satisfy
 $$\inf\limits_{m=1 \cdots M} \Re \big( n_m (x) \big)  \geq n_{min}>0  \quad \text{and}  \quad \Im \left( n_m \right) \geq 0$$
 for almost every $x \in D_m$ where we assume that $n_{min}>1$ or $n_{min}<1$ . Now the radiating scattered field $u^s( \cdot \, , y) \in H^1_{loc} (\R^d)$ given by the point source incident field $u^i( \cdot \, , y)=\Phi( \cdot \, , y)$ is the unique solution to
 \begin{eqnarray}
&&\Delta_x u^s +k^2 n u^s=k^2(1-n)u^i \quad   \textrm{ in } \quad \R^d \label{scalarprob} \\
&&{\partial_r u^s} -iku^s =\mathcal{O} \left( \frac{1}{ r^{(d+1)/2} }\right) \quad \text{ as } \quad r \rightarrow \infty \label{src1}
\end{eqnarray} 
where $r=|x|$ and the radiation condition \eqref{src1} is satisfied uniformly in all directions $\hat{x} = x/|x|$. 
Recall the fundamental solution to Helmholtz equation 
$$\Phi(x,y)= \left\{\begin{array}{lr} \frac{\text{i}}{4} H^{(1)}_0 (k | x-y |) \, \, & \, \text{in} \, \, \R^2 \\
 				&  \\
 \frac{1}{4 \pi} \frac{\text{exp}({ \text{i} k | x-y |}) }{| x-y |}  & \,  \text{in} \,\,   \R^3 ,
 \end{array} \right. $$  
where $H^{(1)}_0$ is the first kind Hankel function of order zero. 
The corresponding total field for the scattering problem is given by $u( \cdot \, , y)=u^s( \cdot \, , y)+u^i( \cdot \, , y)$  satisfies 
$$u^{(-)}=u^{(+)} \quad \textrm{ and } \quad \frac{\partial u^{(-)} }{\partial \nu} =\frac{\partial u^{(+)} }{\partial \nu} \quad \textrm{ on } \partial D_m,$$
where  the superscripts $+$ and $-$ for a generic function indicates the trace on the boundary taken from the exterior or interior of its surrounding domain, respectively. We have that the scattered field is given by the solution to the Lippmann-Schwinger Integral Equation (see for e.g. \cite{coltonkress})
$$u^s(x,y) = k^2 \int_{D} \big( n(z) - 1\big)\Phi(x,z) u(z,y) \, dz.$$
Assume that $\ep$ is sufficiently small such that  
$$ k^2 \left| \int_{D} \big( n(z) - 1\big)\Phi(x,z) \, dz \right| \ll 1. $$
Therefore, we can conclude that the Born approximation, which is the first term in the Neumann series solution to the Lippmann-Schwinger Integral Equation, given by 
\begin{eqnarray}
u_B^s(x,y) = k^2 \int_{D} \big( n(z) - 1\big)\Phi(x,z) u^i  (z,y)  \, dz \label{born}
\end{eqnarray}
is a `good' approximation for the scattered field $u^s(x,y)$. Therefore, we assume that we have the measured near field data that is given by the Born approximation $u_B^s(x,y)$. Let $\Omega$ be the bounded simply connected open set such that satisfies $\overline{D} \subset \Omega$ and $C=\partial \Omega$. Now assume that we have the data set of near field measurements $u_B^s(x,y)$ for all $x,y \in C$. The {\it inverse problem} we consider is to reconstruct the scatterer $D$ and the refractive index $n$ from a knowledge of the Born approximation of the near field measurements.

\subsection{The MUSIC Algorithm}
The MUSIC Algorithm can be considered as the discrete analogue of the factorization method (see for e.g. \cite{MUSIC-LSM}). In particular, we connect the locations of the obstacles $D_m$ given by $\{ z_m  \, : \, m=1, \dots , M\}$ to the range of the so-called multi-static response matrix denoted by ${\bf N}$ that is defined by the near field measurements $u_B^s(x,y)$ for all $x,y \in C$.  To arrive at the MUSIC algorithm, we will need to define our measurements operator. To this end, we let the incident field be give by $u^i(x,y)=\Phi(x,y)$. Notice that by \eqref{born} we obtain that 
\begin{eqnarray*}
&& u_B^s(x,y) = k^2 \int_{D} \big( n(z) - 1\big) \Phi(x,z) \Phi(z,y)  \, dz \\
&& \qquad \qquad =\sum\limits_{m=1}^M k^2  \big( n(z_m) - 1 \big) \ep^d |B_m| \Phi(x, z_m) \Phi(z_m ,y) +o(\ep^d)
\end{eqnarray*}
since $x , y \in C$ the integrand is continuous with respect to the variable $z \in D$. Assume we have a finite number of  incident and observation directions where $N \geq M$, with $M$ being the number of small obstacles and $N$ being the number of incident and observation directions given by $x_i \, ,y_j \in C$. Now define the multi-static response matrix 
$${\bf N}_{i,j} = \sum\limits_{m=1}^M k^2 \big( n(z_m) - 1 \big) \ep^d |B_m| \Phi(x_i , z_m) \Phi(z_m ,y _j).$$
Notice that for $\ep$ sufficiently small we have that the multi-static response matrix is approximated by $u^s_B(x_i\, ,y_j)$ therefore we assume that ${\bf N}$ is known. For convenience, we will assume that the sources and receivers are placed at the same points (i.e. $x_i=y_i$). We now factorize the multi-static response matrix, therefore we define the matrices ${\bf U} \in \C^{N \times M}$ and ${\bf \Sigma} \in \C^{M \times M}$, where the matrix ${\bf U}$ is given by
$${\bf U}_{i, m}=   \Phi( x_i , z_m) $$
and the matrix ${\bf \Sigma} =\text{diag}(\sigma_m)$ where $\sigma_m=k^2 \big( n(z_m) - 1 \big) \ep^d |B_m|\neq 0$. Therefore, we have that ${\bf N}={\bf U}{\bf \Sigma}{\bf U}^{\top}$ and it follows that 
\begin{equation}
{\bf N}{\bf N}^*={\bf U} {\widetilde{\bf \Sigma }} {\bf U}^* \quad \text{ with } \quad {\widetilde{\bf \Sigma}}={\bf \Sigma}{\bf U}^{\top} \, \overline{{\bf U}} {\bf \Sigma}^* . \label{msmfactor}
\end{equation}
Now define the vector ${\bf \phi}_z \in \C^N$ for any point $z \in \R^d$ by 
\begin{eqnarray}
{\bf \phi}_z= \big(\, \Phi(x_1 , z) ,\,  \dots , \Phi(x_N , z) \, \big)^{\top}. \label{vecphi}
\end{eqnarray}
The goal of this section is to show that ${\bf \phi}_z$ is in the range of the matrix $ {\bf N}{\bf N}^*$ if and only if $z \in \{ z_m \, : \, m=1, \dots , M \}$, which is a discrete reformulation of the result of the factorization method stated in the Introduction. Since we are interested in finding obstacles $D_m$, it is sufficient to prove the result only for values of $z \in \Omega$. Recall that without loss of generality, we assume that the sources and receivers are placed at the same locations. We will follow the analytic framework in Section 4.1 of \cite{kirschbook}. We now give a result that can be proven by using standard arguments from linear algebra (see proof of Theorem 3.1 in \cite{elasticmusic} for details). 

\begin{lemma} \label{m672}
Let the matrix ${\bf N}$ have the following factorization ${\bf N}={\bf U}{\bf \Sigma}{\bf U}^{\top}$ where ${\bf U} \in \C^{N \times M}$ and ${\bf \Sigma} \in \C^{M \times M}$ with $N>M$. Assume that 
\begin{enumerate}
\item the matrix ${\bf U}$ has full rank $M$ 
\item the matrix ${\bf \Sigma}$ is invertible 
\end{enumerate}
then Range$\big({\bf U} \big)=$Range$\big( {\bf N}{\bf N}^* \big)$. 
\end{lemma}

We now construct an indicator function derived from Lemma \ref{m672} to reconstruct the locations of the scattering obstacles. To this end, for each sampling point $z \in \Omega$ we will show that the vector $ {\bf \phi}_z$ is in the range of $ {\bf N}{\bf N}^*$ if and only if  $z \in \{ z_m \, : \, m=1, \dots , M\}$. We now prove an auxiliary result that connects the location of the scattering obstacles to the range of the matrix ${\bf U}$.

\begin{theorem} \label{music1}
Assume that the set $S=\{ {x}_i  \in C\, : \, i \in \N \}$ is dense in $C$. Let $z \in \Omega$ then there is a number $N_0 \in \N$ such that for all $N \geq N_0$,  we have that 
\begin{enumerate}
\item the matrix ${\bf U}$ has full rank,
\item ${\bf \phi}_z \in $ Range$\big({\bf U} \big)$ if and only if $z \in \{ z_m \, : \, m=1, \dots , M\}$.
\end{enumerate}
\end{theorem}
\begin{proof}
 It is clear that ${\bf \phi}_{z_m}$ is in the range of ${\bf U}$ since ${\bf \phi}_{z_m}$ is the $m$-th column of ${\bf U}.$ To prove that for $z \in \Omega \setminus  \{ z_m \, : \, m=1, \dots , M\}$ that ${\bf \phi}_{z}$ is not in the range of ${\bf U}$ for some $N$ sufficiently large, we proceed by way of contradiction. Let $z \in \Omega \setminus  \{ z_m \, : \, m=1, \dots , M\}$ and assume that there does not exist such a $N_0$, then we have that there exists $\alpha_m^N$, $\alpha^N \in \C$ such that 
\begin{equation}
 \big| \alpha^N \big|+\sum\limits_{m=1}^M \left| \alpha_m^N \right| =1 \label{equal1}
 \end{equation}
and 
$$ \alpha^N \Phi({ x}_i , z )+\sum\limits_{m=1}^M  \alpha_m^N \Phi({ x}_i , z_m) =0 \quad \text{for all} \quad 1 \leq i \leq N.$$
We then conclude that (up to a subsequence) $\alpha_m^N$, $\alpha^N  \rightarrow \alpha_m$, $\alpha$ as $N \rightarrow \infty$. This gives that due to the density of $S$ and since $\Phi({ x}, z )$ is analytic with respect to $x$ for $z \neq x$, we obtain that 
\begin{equation*}
\alpha \Phi({x}, z )+\sum\limits_{m=1}^M  \alpha_m \Phi({x}, z_m ) =0 \quad \text{for all} \quad {x} \in C.
\end{equation*}
Since the mapping 
\begin{equation*}
x \, \longmapsto \, \alpha \, \Phi({x}, z )+\sum\limits_{m=1}^M  \alpha_m \Phi({x}, z_m )
\end{equation*}
is a radiation exterior solution to the Helmholtz equation in $\R^d \setminus \overline{\Omega}$ we conclude that
$$ \alpha \,  \Phi({x}, z )+\sum\limits_{m=1}^M  \alpha_m \Phi({x}, z_m ) =0  \quad \text{for all} \quad { x} \in \R^d \setminus \{z \}  \cup \{ z_m \, : \, m=1, \dots , M\} $$
by uniqueness to the exterior Dirichlet problem and unique continuation. Now letting $x \rightarrow z, z_m$ gives that $\alpha_m=\alpha=0$ for $1 \leq m \leq M$, but by \eqref{equal1} and the convergence of the coefficients $\alpha_m^N$, $\alpha^N$ as $N \rightarrow \infty$ we have that 
$$ |\alpha|+\sum\limits_{m=1}^M | \alpha_m | =1$$
which gives a contradiction. Moreover, the fact that ${\bf U}$ has full rank is a consequence of the given arguments. 
\end{proof}

Now by combining this with Lemma \ref{m672} and Theorem \ref{music1} we have the following range test which can be used to reconstruct the set $z \in \{ z_m \, : \, m=1, \dots , M\}.$

\begin{theorem} \label{music2}
Assume that the set  $S=\{ {x}_i  \in C\, : \, i \in \N \}$ is dense in $C$.  If $z \in \Omega$ then there is a number $N_0 \in \N$ such that for all $N \geq N_0$  
$${\bf \phi}_z \in \text{Range}\big({\bf N}{\bf N}^* \big)\quad  \text{ if and only if } \quad z \in \{ z_m \, : \, m=1, \dots , M\}.$$ 
\end{theorem}

\noindent Now we are ready to construct an indicator function which characterizes the locations of the small obstacles. To this end, let ${\bf P}: \C^N \mapsto \text{Null}\big({\bf N}{\bf N}^* \big)$ be the orthogonal projection onto $\text{Null}\big({\bf N}{\bf N}^* \big)$. Therefore, by Theorem \ref{music2} we have that 
$${\bf P} {\bf \phi}_{z}=0 \quad \text{ if and only if } \quad z \in \{ z_m \, : \, m=1, \dots , M\}.$$ 
Notice that since ${\bf N}{\bf N}^*$ is a self-adjoint matrix we have that it is orthogonally diagonalizable. So we let ${\bf w}_j$ be the $j$-th orthonormal eigenvector of ${\bf N}{\bf N}^*$, we also let $r$=Rank$\big( {\bf N}{\bf N}^* \big)$. This gives that ${\bf w}_j$ for $r+1 \leq j \leq N$ is an orthonormal basis for $\text{Null}\big({\bf N}{\bf N}^* \big)$, and therefore we have that ${\bf P}$ is given by 
$${\bf P} {\bf \phi}_{z} =\sum\limits_{j=r+1}^N \big( {\bf \phi}_{z} , {\bf w}_j \big)_{\ell^2} {\bf w}_j.$$
This leads to the following result:

\begin{lemma} \label{music3}
Let ${\bf w}_j$ be the $j$-th orthonormal eigenvector of ${\bf N}{\bf N}^*$ and let $r$=Rank$\big( {\bf N}{\bf N}^* \big)$. Assume that the set $S=\{ {x}_i  \in C\, : \, i \in \N \}$ is dense in $C$. If $z \in \Omega$ then there is a number $N_0 \in \N$ such that for all $N \geq N_0$
\begin{equation}
\mathcal{I}(z)= \left[ \sum\limits_{j=r+1}^N \left|\big( {\bf \phi}_{z} , {\bf w}_j \big) \right|_{\ell^2}^2 \right]^{-1} < \infty \quad  \text{ if and only if } \quad z \in \{ z_m \, : \, m=1, \dots , M\}. \label{musicfunc}
\end{equation}
\end{lemma}

\subsection{Numerical validation of the MUSIC algorithm}

We now present a few examples of reconstructing small obstacles using the MUSIC algorithm developed in this section in ${\mathbb R}^2$. 
To simulate the data, we use the Born approximation of the scattered field and a $2d$ Gaussian quadrature to approximate the volume integral. We now let the collection curve $C$ be the boundary of a disk of radius one. In the following calculations we use $32$ different sources and receivers at the points 
$$x_i=y_i=\big(\cos(2\pi(i-1)/32), \sin(2\pi(i-1)/32) \big)^{\top} \quad \text{for} \quad i=1, \dots, 32. $$ 
This leads to an approximated multi-static response matrix 
$${\bf N} \approx \big[ u^{s}_B({x}_i,y_j) \big]_{i,j=1}^{32} \quad \textrm{ such that } \quad u^{s}_B({x},y) = k^2 \sum_p \omega_p \big( n(z_p) - 1\big)\Phi(x,z_p) {\Phi(z_p,y)},$$
where $\omega_p$ and $z_p$ are the weights and quadrature points for the numerical approximation of the volume integral by a  2$d$ Gaussian quadrature method. We give examples with random noise added to the simulated data for $u^{s}_B ({x}_i,y_j)$. The random noise level is given by $\delta$ where the noise is added to the multi-static response matrix such that 
$$\Big[u^{s}_B ({x}_i,y_j) \big( 1 +\delta E_{i,j} \big) \Big]_{i,j=1}^{32} \quad \text{with the random matrix} \quad E \quad \text{satisfying } \quad \| E \|_2 =1.$$
The set $\{ z_m \, : \, m=1, \dots , M\}$ can be visualized by plotting the corresponding indicator function given in \eqref{musicfunc}
$$ \mathcal{I}(z)= \left[ \sum\limits_{j=r+1}^{32}\left|\big( {\bf \phi}_{z} , {\bf w}_j \big) \right|_{\ell^2}^2 \right]^{-1} $$
where $ {\bf w}_j  \in  \C^{32}$ are the eigenvectors of the approximated matrix  ${\bf N}{\bf N}^*$ and $r$=Rank$\big( {\bf N}{\bf N}^* \big)$ is computed by the Matlab command rank. In the following examples, we fix the wave number $k=1$. 

\begin{figure}[H]
\hspace{-0.75in}\includegraphics[scale=0.45]{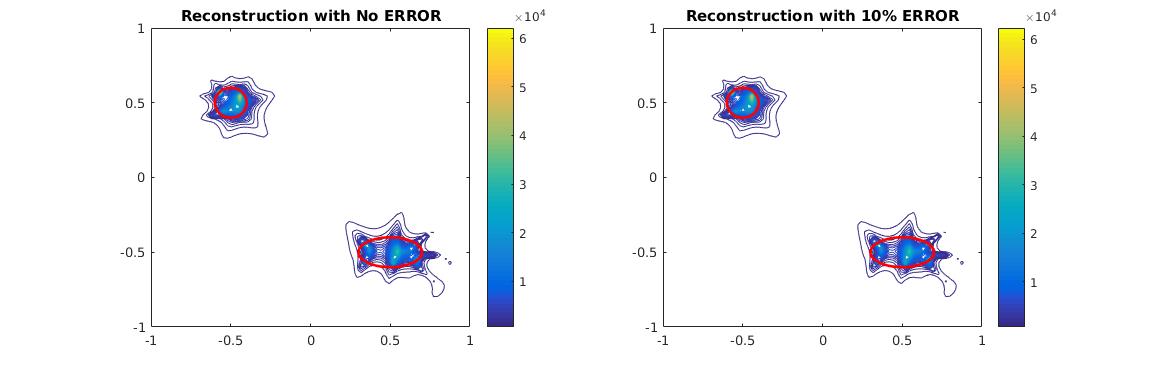}
\caption{ Reconstruction of the location of the two scatterers, the circle of radius 0.2 centered at $(-0.5, 0.5)$ and the ellipse with $a=0.2$ and $b=0.1$ centered at $(0.5, -0.5)$. The refractive index $n=5$ in both scatterers. Contour plot of the indicator function $z \mapsto \mathcal{I}(z)$ where the red lines are the boundary of the scatterer. }
\end{figure}

\begin{figure}[H]
\hspace{-0.7in}\includegraphics[scale=0.53]{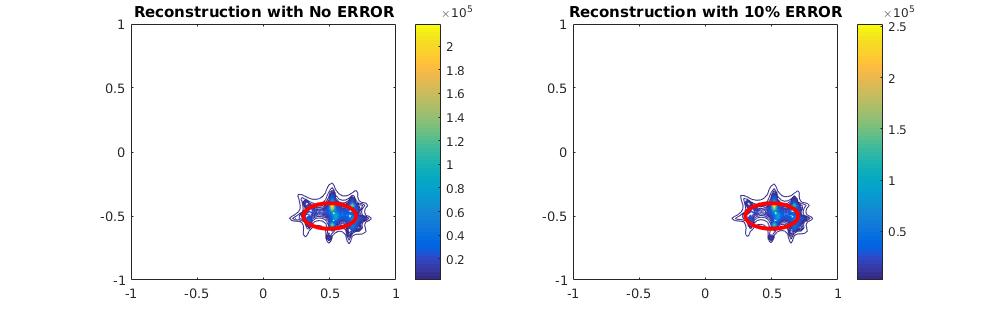}
\caption{ Reconstruction of the location of one scatterer, the ellipse with $a=0.2$ and $b=0.1$ centered at $(0.5, -0.5)$. The refractive index $n=2+\text{i}$ in the scatterer. Contour plot of the indicator function $z \mapsto \mathcal{I}(z)$ where the red lines are the boundary of the scatterer. }
\end{figure}

\begin{figure}[H]
\hspace{-0.7in}\includegraphics[scale=0.48]{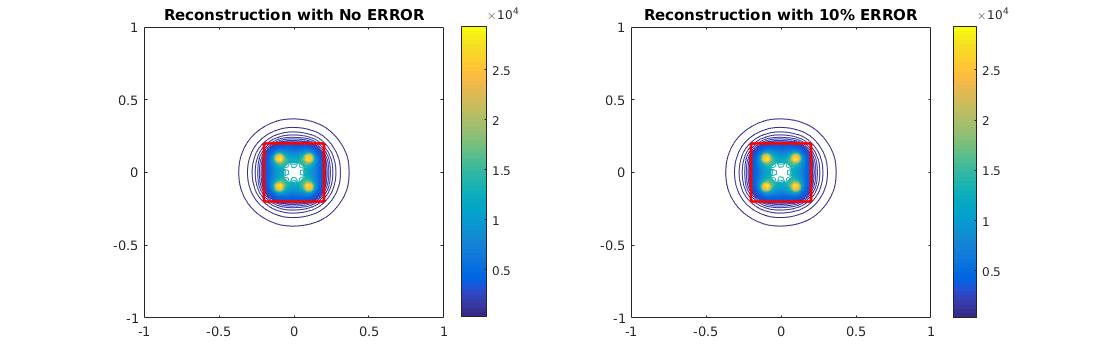}
\caption{ Reconstruction of the location of the a square scatterer, $[-0.2 , 0.2]^2$. The refractive index $n(x_1 , x_2 )=x_1^2+2$ in the scatterer. Contour plot of the indicator function $z \mapsto \mathcal{I}(z)$ where the red lines are the boundary of the scatterer. {\it Remark:} Even though both figures look identical, different data we used in the reconstructions.  }
\end{figure}


{

\subsection{Bayesian approach to approximating the Refractive Index}

In this section, we consider the the second part of our {\it inverse problem} which is to reconstruct (i.e. approximate) the refractive index of the scatterer(s). In particular, the idea is to use a hierarchical Bayesian model for estimating the refractive index by a constant in the scatterer(s) which has been reconstructed, possibly by the MUSIC algorithm proposed in the previous section. Since we assume that $n(x)$ is continuous in the scatterer the small size of the obstacle implies that there will be only small deviations from the value at the center of the obstacle.

Here we assume that we that an approximation of the scatterer is known using the reconstruction from the previous section or by another qualitative method. We use a Bayesian collocation method to approximate the value of the refractive index at the center of the obstacle. For this discussion, assume that the obstacle is centered at the point $x_0$. Rather than solving for the value of the unknown $n(x_0)$ exactly, we solve for a probability distribution of the possible values of $n(x_0)$. To do this, we must assume a prior distribution for possible values of $n(x_0)$ and a data model for the scattered wave, which we will discuss below. Then, Markov chain Monte Carlo methods allow us to sample from the posterior distribution of $n(x_0)$ given the scattered wave and the prior distribution. The mode of the resulting posterior distribution (called the Maximum a posteriori estimator or MAP). The MAP in certain settings can be shown to be equivalent to the standard linear regression estimator or the frequentist statistics maximum likelihood estimator. In this context, our approach provides more information on possible values of $n(x_0)$, and because the scatterer is small with a continuous refractive index. 

To proceed, we define the necessary notation. Assume we have a reconstruction of $D$ via the MUSIC algorithm, denoted $\widehat D$. 
Furthermore, we assume that the measured scattered field denoted $U^s_B$ is given by the Born approximation on $D$, that is
\begin{equation*}
U^s_B (x,y) =k^2 \int_{\widehat D} \big( n(z) - 1\big)\Phi(x,z) {\Phi(z,y)}  \, dz + \mathcal{O}(\delta)
\end{equation*}
where $\delta$ is the amount of error added for the reconstruction of $D$ and standard measurement error. We assume the error term is modeled by a normal distribution with zero mean and standard deviation $\delta$, then we may write the data model, as a statistical distribution where values of $U^s_B$ are drawn from
\begin{equation}
U^s_B  (x,y)  \sim k^2 \int_{ \widehat D} \big( n(z) - 1\big)\Phi(x,z) {\Phi(z,y)}  \, dz  + \mathcal{N}(0,\delta). 
\end{equation}
Now, recall that we have readings of the scattered field $U^s_B$ at points $x_i \in C$ with the location of the point source given by $y_i \in C$ with $C$ being the collection curve discussed in the previous section. This forms the data set $R=\big\{(x_i,y_i,U^s_B(x_i,y_i))\big\}_{i=1}^{N}$ where $N$ indicates the number of sources and receivers.  

We now motivate the subsequent Bayesian model. Since the goal to reconstruct the value of the refractive index at the center of the scatterer, using Taylor's Theorem  we have that $n(x)=n(x_0) + \mathcal{O}(h)$ where $h=|\widehat D|$ is small. Therefore, the constant term $n(x_0)$ of the Taylor polynomial centered at the center of $\widehat D$ is approximately close to all values of $n(x)$ on $\widehat D$. Based on these assumptions, we define the following auxiliary function
\begin{equation}
\eta(x) = n(x) -1 = n(x_0) -1 + \mathcal{O}(h).
\end{equation}
Therefore, we describe the following Bayesian model to solve for the posterior distribution of $n(x_0)$ given readings of the scatterer $R$, i.e. we solve for the conditional pdf function $P(\gamma | R,w_p,z_p)$
\begin{equation*}
U^s_B(x,y) \sim \mathcal{N}(\mu , \delta^2 ) \quad  \text{with} \quad  \mu (x,y) = k^2 \sum_p \omega_p \eta(z_p) \Phi(x,z_p) \Phi(z_p,y)
\end{equation*}
where $\omega_p$ and $z_p$ are the weights and quadrature points for the numerical approximation of the volume integral over $\widehat D$. Here we assume that the error from the numerical integration is negligible since one can take a quadrature rule to approximate the integral to some preferred tolerance.
By the definition of $\eta$ we have that 
\begin{equation*}
\eta(z) \sim \mathcal{N}(\gamma , h^2) \quad \text{ where } \quad \gamma \sim \mathcal{U}(-\infty, \infty).
\end{equation*}

We may solve the model for posterior distribution for $\gamma$, which is given according to the proportion:
\begin{equation}
P \big(\gamma ,\eta| R, \omega_p , z_p \big) \propto \ell \big( \mu | R, \gamma, \eta, \omega_p , z_p \big) P(\eta| \gamma, \omega_p , z_p) P(\gamma),
\end{equation}
where $\ell$ is the likelihood function arising from the data model for $y_i$ and $P(\nu | \gamma), P(\nu)$ are the assumed priors given above. One may marginalize over $\eta$ obtain $P \big( \gamma | R,\omega_p , z_p \big)$ if desired. By using the Metropolis-Hastings algorithm, we can derive this posterior distribution, and thus values for $\gamma$. We note that this model is hierarchical with care taken to model the sources of error. This should lead to a more realistic solution.

For our calculation, we use the implementation in the Python package PyMC3 \cite{pymc3}. Therefore we will have $\gamma$ which is the average value of $\eta$ which will be the approximation on $n(x_0) - 1$. For our numerical example, it is a common practice to use a normal distribution with a large standard deviation ($10^5$) rather than a flat prior, $\mathcal{U}(-\infty,\infty)$. For our example we consider the square scatterer $[-0.2 , 0.2]^2$ with refractive index $n(x_1 , x_2)=x^2_1 + 2$ where we add $15\%$ noise to the measurements.  For the reconstruction we use the domain actual scatterer $D=[-0.2 , 0.2]^2$ as well as  $\widehat D=[-0.265 , 0.265]^2$ which is an approximation taken from the MUSIC Algorithm's reconstruction of the obstacle. Using the methods described in this section we are able to determine the distribution of possible values of $\gamma$ which should approximately be  $n(x_0) -1$. \\ 

}
\begin{figure}[H]
\hspace{-0.4in}
\begin{center}
\includegraphics[scale=0.6]{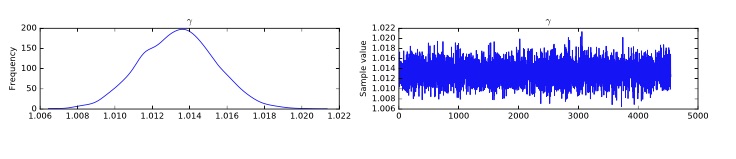}
\caption{Recovery from measurements in the case when $n(x_1 , x_2)=x^2_1 + 2$. The distribution is the possible values of $\gamma \approx n(x_0) -1 = 1$ in the reconstructed scatterer $ D=[-0.2 , 0.2]^2$.}
\end{center}
\end{figure}

\begin{figure}[H]
\hspace{-0.4in}
\begin{center}
\includegraphics[scale=0.6]{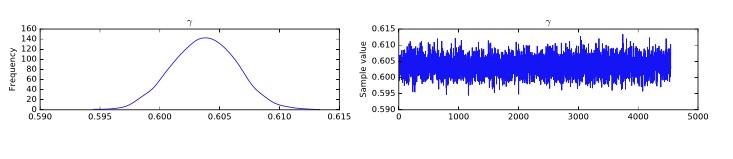}
\caption{Recovery from measurements in the case when $n(x_1 , x_2)=x^2_1 + 2$. The distribution is the possible values of $\gamma \approx n(x_0) -1 = 1$ in the reconstructed scatterer $\widehat D=[-0.265 , 0.265]^2$.}
\end{center}
\end{figure}

Similar approaches utilizing Bayesian statistics and the Born approximation have been used when applying compressive sensing techniques to inverse scattering \cite{cs1},\cite{cs2},\cite{cs3},\cite{cs4}. In those context, there is an a-priori assumption on the sparsity of the scatterer. The approach has also been applied to continuous random media \cite{cs5}. However, in this case, the scatterer is not sparse and it is not a continuous random medium. Moreover, we extend the approach by using the MUSIC algorithm to find the support of the scatterer, and use that information in our recovery of the refractive index.

\section{ {Sampling methods for anisotropic scatterers } }
\subsection{Scattering by extended anisotropic scatterers}\label{problem}
In this section we formulate the direct and inverse time-harmonic scattering problems under consideration. To this end, we let the point source located at the point $y \in C$ be given by the fundamental solution to Helmholtz equation $\Phi(x,y)$  
where $C$ is assumed to be a closed curve (in $\R^2$) or surface (in $\R^3$) that is class $\mathcal{C}^2$ or sufficiently smooth, such that $H^{1/2}(C)$ is compactly embedded in $L^2(C)$. We consider the scattering of the (non-standard) incident field $u^i( \cdot \, , y)=\overline{ \Phi( \cdot \, , y) }$ for $y \in C$ by a penetrable anisotropic obstacle. Assume that the obstacle $D \subset \R^d$ (for $d=2,3$) is a bounded simply connected open set having piece-wise smooth boundary $\partial D$ with $\nu$ being the unit outward normal to the boundary.

We assume that the constitutive parameters of the scatterer $D$ are represented by a  symmetric matrix $\tilde A \in  \mathcal{C}^{2}  \left( D, \C^{d \times d} \right)$ and a scalar function $\tilde n \in \mathcal{C}(D)$  such that  
$$\overline{\xi}\cdot \Re \left(\tilde A(x) \right) \xi\geq a_{min} |\xi|^2>0 \quad \text{and}  \quad \overline{\xi}\cdot \Im \left( \tilde A(x) \right)\xi \leq 0,$$
where as 
 $$\Re \big( \tilde n(x) \big) \geq n_{min}>0  \quad \text{and}  \quad \Im \big( \tilde n(x) \big) \geq 0$$
  for almost every $x\in D$ and all $\xi\in {\mathbb C}^d$. Outside the obstacle $D$, the material parameters are homogeneous isotropic with refractive index scaled to one. We denote by $A$ and $n$ the constitutive parameters of the homogenous background with the anisotropic obstacle $D$ in ${\mathbb R}^d$ by
 $$A(x)= \tilde A(x) \chi_D +I (1-\chi_D) \quad  \text{and} \quad n(x)= \tilde n(x) \chi_D +1 (1-\chi_D) $$
 where $I$ denotes the $d \times d$ identity matrix and $\chi_D$ is the characteristic function on $D$. Note that the support of the contrasts $A-I$ and $n-1$ is $\overline{D}$. The radiating scattered field $u^s( \cdot \, , y)$ given by the (non-standard) point source $u^i( \cdot \, , y)$ satisfies 
\begin{eqnarray}
&&\nabla_x \cdot A \nabla_x u^s +k^2 n u^s=\nabla_x \cdot (I-A)  \nabla_x u^i +k^2(1-n)u^i \quad   \textrm{ in } \,  \R^d \label{forwardprob} \\
&&{\partial_r u^s} -iku^s=\mathcal{O} \left( \frac{1}{ r^{(d+1)/2} }\right) \quad \text{ as } \quad r \rightarrow \infty , \label{src}
\end{eqnarray} 
where $r=|x|$ and \eqref{src} is satisfied uniformly in all directions $\hat{x}=x/|x|$. It can be shown that the scattering problem \eqref{forwardprob}-\eqref{src} has a unique solution $u^s \in H^1_{loc}(\R^d)$ by using a variational approach (see for e.g. \cite{p1}). The total field corresponding to the scattering problem \eqref{forwardprob} is defined as $u( \cdot \, , y)=u^s( \cdot \, , y)+u^i( \cdot \, , y)$ and satisfies 
$$u^{(-)}=u^{(+)} \quad \textrm{ and } \quad \frac{\partial u^{(-)} }{\partial \nu_{A}} =\frac{\partial u^{(+)} }{\partial \nu} \quad \textrm{ on } \partial D ,$$ 
where  the superscripts $+$ and $-$ for a generic function indicates the trace on the boundary taken from the exterior or interior of the domain, respectively. Now assume that we have the data set of near field measurements  $u^s(x,y)$ for all $x,y \in C$ and define the Near Field operator 
$$  {N} : L^2(C) \longmapsto L^2(C) \quad  \text{given by} \quad  ({N}g)(x) = \int_{C} u^s(x,y) g(y) \, ds_y.$$ 
The {\it inverse problem} we consider is to reconstruct the scattering obstacle $D$ from a knowledge of $u^s(x,y)$ for all $x, \, y \in C$. We will let $\Omega$ be the bounded simply connected open set such that $\overline{D} \subset \Omega$ and $C=\partial \Omega$. 

In the coming sections, we will see that to obtain a symmetric factorization of the near field operator we require that the incident wave be the complex conjugate of a point source. However, it is known that these non-physical sources can be approximated by a linear combination of physical point sources (see \cite{pointsource}). In \cite{nf-fm-isotropic}, the authors are able to avoid using non-physical sources but must construct the so-called outgoing-to-incoming operator. If $C$ is given by the boundary of a disk/sphere centered at the origin, then the outgoing-to-incoming operator is given by a boundary integral operator that does not depend on the type of scatterer. 

\subsection{Factorization of the near field operator}
\noindent This section is dedicated to constructing a suitable factorization of the near field operator $N$ so that we are able to use the theoretical results in {\cite{factor-music} }and/or \cite{armin}, in order to derive {indicator functions} for the support of $D$ in terms of the near field measurements. {Here we consider a modified linear sampling method as well as the factorization method and just as in \cite{arens} we will us the range characterization from the factorization method to show that these two methods give equivalent indicator functions.} To this end, motivated by the source problem given in equation (\ref{forwardprob})-(\ref{src}) for the scattered field due to the (non-standard) point source, we consider the problem of finding  $u \in H^1_{loc}(\R^d)$ for a given $v \in H^1(D)$ such that 
\begin{eqnarray}
&&\nabla \cdot A \nabla u +k^2 n u=\nabla \cdot (I-A)  \nabla v +k^2(1-n)v \quad   \textrm{ in } \,  \R^d \label{sourcep1} \\
&&\hspace{1cm} {\partial_r u^s} -iku^s =\mathcal{O} \left( \frac{1}{ r^{(d+1)/2} }\right) \quad \text{ as } \quad r \rightarrow \infty. \nonumber
\end{eqnarray}
It can be shown that for any given $v$, the solution operator that maps $v \mapsto u$ is a bounded linear mapping from $H^1(D)$ to $H^1_{loc}(\R^d)$. At this point, let us recall the exterior Dirichlet-to-Neumann mapping for Helmholtz equation 
$$\mathbb{T}_k: H^{1/2}(\partial B_R) \longmapsto H^{-1/2}(\partial B_R) \quad \text{given by} \quad \mathbb{T}_k f=\frac{\partial \varphi}{\partial \nu} \quad \text{on} \, \, \, {\partial B_R},$$ 
where $\varphi \in H^1_{loc}(\R^d \setminus \overline{B_R} )$ satisfies
\begin{align*}
&\Delta \varphi +k^2 \varphi=0 \quad  \quad \text{in } \R^d \setminus \overline{B}_R  \quad \text{and} \quad \varphi=f    \quad \text{on } \partial B_R 
\end{align*}
along with the radiation condition \eqref{src} where $B_R=\{ x \in \R^d : |x| < R\}$. With the help of the Dirichlet-to-Neumann operator we can write (\ref{sourcep1}) in its equivalent variational form: find $u\in H^1(B_R)$ such that  
 \begin{align}
& - \int\limits_{B_R} A\nabla u \cdot  \nabla \overline{\varphi} -k^2 nu  \overline{\varphi}\, dx\, + \int\limits_{\partial B_R}  \overline{\varphi}\, \mathbb{T}_k u \, ds \nonumber\\
&\hspace{3cm}=-\int\limits_{D} (I-A)\nabla v \cdot  \nabla \overline{\varphi} -k^2 (1-n)v \overline{\varphi}\, dx, \quad \forall \varphi \in H^1(B_R) ,  \label{hadjointform0}
 \end{align}
where the boundary integral over $\partial B_R$ is understood as the dual pairing between $H^{1/2}(\partial B_R)$ and $H^{-1/2}(\partial B_R)$. 
It is clear by  well-posedness that if we take $v=\overline{ \Phi( \cdot \, , y) } \big|_{D}$ then we have that the scattered field $u^s(\cdot \, y)=u$ for all $x \in \R^d$. Now define the source to trace operator 
$$G: H^1(D) \longmapsto L^2(C) \quad  \text{given by} \quad Gv=u \big|_{C}$$
as well as the bounded linear operator 
$$H : L^2(C)  \longmapsto H^1(D) \quad  \text{given by} \quad  (Hg)(x) = \int_{C} \overline{ \Phi( \cdot \, , y) }  g(y) \, ds_y \big|_{D} .$$
By superposition, we have that near field operator is given by $N=GH.$ In order to use the theory developed in \cite{kirschbook} and/or \cite{armin}, we must construct a symmetric factorization of the near field operator. Therefore, we compute the adjoint of the operator $H$ which is given in the following result. 

 \begin{theorem}\label{hadjoint}
 The operator $H^*: H^1(D) \longmapsto L^2(C)$ is given by 
 $$ H^* v= \tilde{v} \big|_{C} \quad \text{for all}  \quad v \in H^1(D),$$
 where $\tilde{v} \in H^1_{loc}(\R^d)$ is the unique radiating solution to 
 \begin{eqnarray}
 -\int\limits_{B_R} \nabla \tilde{v} \cdot  \nabla \overline{\varphi} -k^2 \tilde{v}  \overline{\varphi}\, dx +\int\limits_{\partial B_R} \overline{\varphi} \mathbb{T}_k  \tilde{v}\, ds =(v,\varphi)_{H^1(D)} \label{hadjointform}
\end{eqnarray}
 for all $\varphi \in H^1(B_R)$.
 \end{theorem}
\begin{proof}
Let the radius of $B_R$ be sufficiently large so that $\overline{ \Omega} \subset B_R$. Given any $v \in  H^1(D)$ we can construct a unique radiating field  $\tilde{v} \in H^1_{loc}(\R^d)$ that satisfies \eqref{hadjointform}. We now let $w \in H^1_{loc}(\R^d)$ be defined as
$$w=\int_{C} \overline{ \Phi( x \, , y) }  g(y) \, ds_y \quad \text{for} \, \, \, \, x \in \R^d.$$
Therefore, we have that 
\begin{align*}
&(H^* v,g)_{L^2(C)} = (v,Hg)_{H^1(D)} \\
					    &\hspace{1.5cm}= -\int\limits_{B_R } \nabla \tilde{v} \cdot  \nabla \overline{w} -k^2 \tilde{v}  \overline{w}\, dx +\int\limits_{\partial B_R} \overline{w} \mathbb{T}_k  \tilde{v}\, ds \\
					    &\hspace{1.5cm}= -\int\limits_{B_R \setminus \overline{ \Omega} } \nabla \tilde{v} \cdot  \nabla \overline{w} -k^2 \tilde{v}  \overline{w}\, dx  - \int\limits_{ { \Omega} } \nabla \tilde{v} \cdot  \nabla \overline{w} -k^2 \tilde{v}  \overline{w}\, dx +\int\limits_{\partial B_R} \overline{w} \mathbb{T}_k \tilde{v}\, ds. 
\end{align*}
Using integration by parts on the volume integrals and the fact that $w$ solves Helmholtz equation in $\R^d \setminus C$ we have that 
\begin{eqnarray*}
&&(H^* v,g)_{L^2(C)} =  \int\limits_{C } \tilde{v} \left( \frac{\partial}{\partial \nu} \overline{w}^{(-)} -  \frac{\partial}{\partial \nu} \overline{w}^{(+)} \right) ds +\int\limits_{\partial B_R} \overline{w}  \frac{\partial}{\partial \nu} \tilde{v} -  \tilde{v} \frac{\partial}{\partial \nu} \overline{w} \, ds. \\
\end{eqnarray*}
Now using the jump relations for the normal derivative of the single layer potential for Helmholtz equation gives that 
$$ \frac{\partial}{\partial \nu} \overline{w}^{(-)} -  \frac{\partial}{\partial \nu} \overline{w}^{(+)} = \overline{g} \quad \text{on} \quad C.$$
Notice that both $\tilde{v}$ and $\overline{w}$ satisfy the radiation condition \eqref{src} therefore letting $R \rightarrow \infty$ gives that 
$$(H^* v,g)_{L^2(C)} = ( \tilde{v} , g )_{L^2(C)},$$
proving the claim. 
\end{proof}

Just as in \cite{fmdefect} notice that \eqref{sourcep1}, implies that 
\begin{align}
&\Delta u +k^2 u=\nabla \cdot (I-A)  \nabla (v+u) +k^2(1-n)(v+u) \,  \textrm{ in } \,  \R^d , \label{sourcepw}
\end{align}
where $u \in H^1_{loc}(\R^d)$ is the solution to \eqref{sourcep1} for a given $v \in H^1(D)$. The variational form of \eqref{sourcepw} is given by 
\begin{align*}
&\hspace{-1.5cm}-\int\limits_{B_R} \nabla u \cdot  \nabla \overline{\varphi} -k^2 u \overline{\varphi}\, dx + \int\limits_{\partial B_R} \overline{\varphi} \mathbb{T}_k u \, ds  = - \int\limits_{D} (I-A) \nabla (v+u) \cdot \nabla \overline{ \varphi } \, dx \\
&\hspace{2cm} +k^2  \int\limits_{D} (1-n)(v+u) \overline{ \varphi } \, dx \qquad \forall \varphi \in H^1(B_R).
\end{align*}
Now appealing to the Riesz representation theorem, we can define the bounded linear operator $T: H^1(D) \longmapsto H^1(D)$  such that for all $v \in H^1(D)$
\begin{equation} 
(T v, \varphi)_{ H^1(D) }= - \int\limits_{D} (I-A) \nabla (v+u) \cdot \nabla \overline{ \varphi }-k^2 (1-n)(v+u) \overline{ \varphi } \, dx.   \,\,\,  \label{tk}
\end{equation}
Notice for $u\in H^1_{loc}(\R^d)$ the unique solution to \eqref{sourcep1} that by the definition of $G$ we have that $u \big|_{C}=Gv$. Now by the definition of $H^*$ and $T$ we conclude that $u \big|_{C} = H^*T v$. We conclude that $Gv = H^*T v$ for all $v \in H^1(D)$.  Now recalling that $N=GH$ gives the desired factorization. 

\begin{lemma}
The near field operator $N: L^2(C) \longmapsto  L^2(C)$ for the scattering problem \eqref{forwardprob}-\eqref{src} has the factorization $N=H^*TH$. 
\end{lemma}

\subsection{The solution to the inverse problem}
\noindent The main goal of this section is to connect the support of $D$ to the range of an operator defined by the measured near field operator. We make this connection by analyzing the factorization of the near field operator developed in the previous section. We recall from the previous section that  we have the following factorization $N=H^*TH$. Under some appropriate assumptions on $H$ and $T$, the factorization method gives that Range$(H^*)$=$\text{Range} ( N_{\sharp}^{1/2} )$. 
{Here $N_{\sharp}$ will be given by $N_{\sharp}= |\Re({N})| +|\Im({N})| $ for a non-absorbing material or $N_{\sharp}= \sigma \Im({N})$ with $|\sigma|=1$ for an absorbing material where 
$$\Re(N)=\frac{N+N^*}{2}  \quad \text{and}  \quad \Im(N)=\frac{N-N^*}{2\text{i}} .$$
}
Furthermore for a generic self-adjoint compact operator $B$ on a Hilbert space $U$, by appealing to the Hilbert-Schmidt theorem we can define $|B|$ in terms of the spectral decomposition such that 
$$|B|x = \sum |\lambda_j | (x, \psi_j) \psi_j$$
for all $x\in U$ where  $(\lambda_j , \psi_j) \in \R \times U$ is the  orthonormal eigensystem of $B$. 

\subsubsection{Analysis of the operator $H$} 
Now we turn our attention to showing that the operator $H$ has the properties needed to apply Theorems \ref{fmthm2}. To this end, let us define the interior transmission problem as finding a pair $(w,v) \in H^1(D) \times H^1(D)$ such that for a given pair $(f,h) \in H^{1/2}(\partial D) \times H^{-1/2}(\partial D)$ satisfies
\begin{align}
&\nabla \cdot A \nabla w +k^2 n w=0 \quad \text{and} \quad \Delta  v + k^2 v=0  \quad \text{ in } \,\,   D \label{fmitp1} \\
&w-v= f \quad \text{and} \quad  \frac{\partial w}{\partial \nu_{A}}-\frac{\partial v}{\partial \nu}=  h \quad \text{ on }\, \,  \partial D \label{fmitp3}
\end{align}
\begin{definition} \label{trig} {\em The values of $k \in \C$  for which the homogeneous interior transmission problem$\big($i.e. \eqref{fmitp1}-\eqref{fmitp3} with  $(f,h)=(0,0)\big)$, has nontrivial  solutions are called transmission eigenvalues for $A, \, n$ and $D$.}
\end{definition}

 It is known that under appropriate assumptions on the coefficients that the interior transmission problem \eqref{fmitp1}-\eqref{fmitp3} is Fredholm of index zero, i.e if $k$ is not a transmission eigenvalue then \eqref{fmitp1}-\eqref{fmitp3} is well-posed (see for e.g. \cite{Bon-Che-Had-2011} and \cite{chh}). The following results are know for the transmission eigenvalue problem (see \cite{Bon-Che-Had-2011}, \cite{p1} and \cite{cakonikirsch})
\begin{enumerate}
\item If the scatterer is absorbing such that $\Im(A)\leq0$ and $\Im (n)>0$ in $D$ then there are no real transmission eigenvalues.
\item If the scatterer is non-absorbing (i.e. $\Im(A)=0$ and $\Im(n)=0$) then the set of real transmission eigenvalues is at most discrete provided that $A-I$ is positive (or negative) definite uniformly in $D$.
\item If the scatterer is non-absorbing then the set of real transmission eigenvalues is at most discrete provided that $A-I$ is positive (or negative) definite uniformly and $n-1$ is positive (or negative) in a neighborhood of the boundary $\partial D$.
\end{enumerate}

\begin{assumption}\label{note}
The wave number $k \in \R$ is not a transmission eigenvalue for \eqref{fmitp1}-\eqref{fmitp3}. 
\end{assumption}

We are now ready to connect the support of $D$ to the range of $H^*$. Since we assume that the obstacle $D$ is contained in the interior of $\Omega$  we only need to consider sampling points $z \in \Omega$. 

\begin{theorem} \label{rangeh}
Under Assumption \ref{note}, we have that for $z \in \Omega$ 
$$\Phi( \cdot \, , z) \in \text{Range}(H^*)\quad  \text{ if and only if} \quad z \in D.$$
\end{theorem}
\begin{proof}
We start with the case where $z \in \Omega \setminus \overline{D}$ and assume that $v \in H^1(D)$ is such that $H^* v=\Phi( \cdot \, , z)$. By the definition of $H^*$ there exists a $\tilde{v} \in H^1_{loc}(\R^d)$ satisfying \eqref{hadjointform} which implies that  
$$\Delta \tilde{v}+k^2  \tilde{v}=0 \quad \text{in}\quad \R^d \setminus \overline{D} \quad \text{and} \quad \tilde{v} =\Phi( \cdot , z) \quad \text{on} \, \, C $$
along with the Sommerfeld radiation condition \eqref{src}. Now since $z \in \Omega \setminus \overline{D}$, we have that $\Phi( \cdot \, , z)$ is in $H^1_{loc}(\R^d \setminus \overline{\Omega})$, and by the uniqueness of solutions to the exterior Dirichlet problem for Helmholtz equation, we have that $\tilde{v} =\Phi( \cdot , z)$ in $\R^d \setminus \overline{\Omega}$. Therefore by Holmgren's theorem and unique continuation, we have that $\tilde{v}=\Phi( \cdot , z) $ in $\Omega \setminus (\overline{D} \cup \{z\})$. We then conclude that $\Phi( \cdot \, , z) \notin \text{Range}(H^*)$ since 
$$\Phi( \cdot , z) \notin H^1\big(B_r(z)\big) \quad \text{and} \quad \tilde{v} \in H^1\big(B_r(z)\big),$$
for any disk $B_r(z)$ centered at $z$ of radius $r>0$.

Now assume that $z \in D$ then we have that $\Phi( \cdot \, , z)\in H^1_{loc}(\R^d \setminus \overline{D})$. Since $k$ is not a transmission eigenvalue in $D$ we have that there is a unique solution to the interior transmission problem \eqref{fmitp1}-\eqref{fmitp3} with 
$$(f,h)=\left( \Phi( \cdot \, , z) \, ,\,  \displaystyle{\frac{\partial}{\partial \nu}} \Phi( \cdot \, , z)  \right) \quad \text{denoted} \quad (w_z,v_z) \in H^1(D) \times H^1(D).$$ 
Now define the function  
$$u_z :=\left\{ \begin{array}{rcl} w_z-v_z & \mbox{in} & D \\   
\\
\Phi( \cdot \, , z) & \mbox{in} &\R^d \setminus \overline{D}.  \end{array} \right. $$
Therefore we have that $u_z \in H^1_{loc}(\R^d)$ with $u_z \big|_C=\Phi( \cdot \, , z) \big|_C$ such that  
\begin{eqnarray*}
&&\Delta u_z +k^2 u_z=\nabla \cdot (I-A)  \nabla w_z +k^2(1-n) w_z  \,  \textrm{ in } \,  \R^d. 
\end{eqnarray*}
The variational form of the above equation gives that for all $\varphi \in H^1(B_R)$ with $\mathbb{T}_k$ being the Dirichlet to Neumann mapping 
 \begin{align}
&\hspace{-2cm} -\int\limits_{B_R} \nabla u_z \cdot  \nabla \overline{\varphi} -k^2  u_z  \overline{\varphi}\, dx +\int\limits_{\partial B_R} \overline{\varphi} \mathbb{T}_k u_z \, ds = \nonumber \\
&\hspace{2cm}- \int\limits_{D} (I-A) \nabla w_z \cdot \nabla \overline{ \varphi }   - k^2 (1-n) w_z \overline{ \varphi } \, dx. \label{theta}
 \end{align}
 Now let $\phi_z \in H^1(D)$ be defined from the right hand side of (\ref{theta}) by appealing to the Riesz representation theorem, hence we have  
$$-\int\limits_{B_R} \nabla u_z \cdot  \nabla \overline{\varphi} -k^2  u_z  \overline{\varphi}\, dx +\int\limits_{\partial B_R} \overline{\varphi} \mathbb{T}_k u_z \, ds =(\phi_z,\varphi)_{H^1(D)} .$$
Thus we now conclude that $ H^* \phi_z=\Phi( \cdot ,z)$ by the definition of $H^*$ giving the result. 
\end{proof}

\noindent  Next we show that $H$ satisfies that necessary properties to apply Theorems \ref{fmthm2}.

\begin{theorem} \label{rangeh1}
The operator $H :L^2(C) \longmapsto H^1(D)$ is compact and injective. 
\end{theorem}
\begin{proof}
Notice that the mapping $v \in H^1(D) \mapsto \tilde{v} \in H^1_{loc}(\R^d)$ is bounded and by the Trace theorem the mapping $\tilde{v} \in H^1_{loc}(\R^d) \mapsto \tilde{v} \big|_C \in H^{1/2}(C)$ is also bounded. Now using that $H^{1/2}(C)$ is compactly embedded in $L^2(C)$ gives that $H^*$ is compact and therefore so is $H$. 

Now assume that $g \in$ Null$(H)$, therefore we define 
$$w=\int_{C} { \Phi( x \, , y) }  \overline{g(y)} \, ds_y \quad \text{for} \,\,\,\, x \in \R^d.$$
We have that $w$ is a radiating solution to Helmholtz equation in $\R^d \setminus C$, and by the definition of $H$ we have that $w=0$ in $D$, and therefore by unique continuation we have that $w=0$ in $\Omega$. The continuity of the trace of $w$ on $C$ as well as the uniqueness of solutions to the exterior Dirichlet problem for Helmholtz equation implies that $w=0$ in $B_R$ for any radius sufficiently large. Now assume that $\overline{\Omega} \subset B_R$ using the jump relations for the normal derivative for the single layer potential for Helmholtz equation on $C$ gives that 
$$0=\left( \frac{\partial}{\partial \nu} {w}^{(-)} -  \frac{\partial}{\partial \nu} {w}^{(+)} \right)= \overline{g}$$
proving the claim. 
\end{proof}

\subsubsection{Analysis of the operator $T$} 
Next we analyze the properties of the middle operator $T$ defined variationally by equation \eqref{tk}. The operator $T$ is defined similarly to the operator used in the factorization of the far field operator discussed in \cite{fmdefect}, this gives that the analysis 
{
to show that $T$ satisfies the assumptions in Theorem \ref{fmthm2} provided that the material is non absorbing (see \cite{fmdefect} for details). The following result can be obtained from the analysis in \cite{fmdefect}. 

\begin{theorem} \label{T-analysis}
Assume $\Im(A) =0$, $\Im(n) =0$ in $D$ and $A-I$ is positive (or negative) definite  uniformly in $D$ and $n-1$ is negative (or positive) uniformly in $D$. Then the operator $T : H^1(D) \longmapsto  H^1(D)$ satisfies that following:
\begin{enumerate}
\item $T$ is injective 
\item $\Im(T)$ is compact and non-positive
\item  If $A-I$ is uniformly positive definite in $D$ then $\Re(T)$ is the sum of a compact operator and a self-adjoint coercive operator
\item If $A-I$ uniformly negative definite in $D$ then $-\Re(T)$ is the sum of a compact operator and a self-adjoint coercive operator
\end{enumerate}
\end{theorem}

The above result gives that the operator $T$ for a non absorbing media satisfies the assumptions of Theorem \ref{fmthm2}. 
%
%
We now show that for an absorbing media that the imaginary part of $T$ is coercive. } 
Recall that if $u\in H^1_{loc}(\R^d)$ is a radiating solution to the exterior Helmholtz equation then we have that $u$ satisfies that following expansion 
\begin{align} 
u(x)=\frac{\text{e}^{\text{i}k|x|}}{|x|^{\frac{d-1}{2}}} \left\{u^{\infty}(\hat{x}) + \mathcal{O} \left( \frac{1}{|x|}\right) \right\}\; \textrm{  as  } \;  |x| \to \infty ,\label{ffpattern}
\end{align}
where $\hat x:=x/|x|$ and $u^{\infty}(\hat x) $ is the corresponding far field pattern (see \cite{coltonkress} in ${\mathbb R}^3$ and \cite{p1} in ${\mathbb R}^2$). 

{
\begin{theorem} \label{tkanalysis}
 Assume that $\Im(A)$ is uniformly negative definite and $\Im(n) $ is uniformly positive in $D$. Then the operator $-\Im(T) : H^1(D) \longmapsto  H^1(D)$ is coercive.
 \end{theorem}
 }
 \begin{proof}
By \eqref{sourcep1} we have that for any $v_j \in H^1(D)$ there is a unique solution $u_j \in H^1_{loc}(\R^d)$. Now letting $\phi_j=v_j+u_j$ and using \eqref{tk} we have that 
\begin{align*}
&\hspace*{-1cm}(T v_1, v_2)_{ H^1(D) }= - \int\limits_{D} (I-A) \nabla \phi_1 \cdot \nabla \overline{ (\phi_2-u_2) }-k^2 (1-n)\phi_1 \overline{ (\phi_2-u_2 )} \, dx\\
&\hspace*{1cm}= - \int\limits_{D} (I-A) \nabla \phi_1 \cdot \nabla \overline{ \phi_2 }-k^2 (1-n)\phi_1 \overline{ \phi_2} \, dx\\
&\hspace*{3cm}+ \int\limits_{D} (I-A) \nabla \phi_1 \cdot \nabla \overline{ u_2 }-k^2 (1-n)\phi_1 \overline{ u_2} \, dx.
\end{align*}
Now by equation \eqref{sourcep1}, we have that  
\begin{align*}
\Delta u_1 +k^2  u_1=\nabla \cdot (I-A)  \nabla \phi_1 +k^2(1-n)\phi_1 \quad \text{ in } \quad    \R^d ,
\end{align*}
and therefore multiplying by $\overline{u_2}$ and integrating by parts over $B_R$ such that $D \subset B_R$ gives that 
\begin{align*}
&\hspace*{-0.5cm}-\int\limits_{B_R} A\nabla u_1 \cdot  \nabla \overline{u_2} -k^2 n u_1 \overline{u_2}\, dx +\int\limits_{\partial B_R} \overline{u_2}  \frac{\partial u_1}{\partial \nu}\, ds \\
 &\hspace*{5cm}= - \int\limits_{D} (I-A) \nabla \phi_1 \cdot \nabla \overline{ u_2 } - k^2  \int\limits_{D} (1-n) \phi_1 \overline{ u_2 } \, dx. 
\end{align*}
This gives that  
\begin{align}
&\hspace*{-1cm}(T v_1, v_2)_{ H^1(D ) }= - \int\limits_{D} (I-A) \nabla \phi_1 \cdot \nabla \overline{ \phi_2 }-k^2 (1-n)\phi_1 \overline{ \phi_2} \, dx \nonumber \\
&\hspace*{3cm}+\int\limits_{B_R} \nabla u_1 \cdot  \nabla \overline{u_2} -k^2  u_1 \overline{u_2}\, dx -\int\limits_{\partial B_R} \overline{u_2}  \frac{\partial u_1}{\partial \nu}\, ds.    \label{vart1}
\end{align}
Now take the imaginary part of equation (\ref{vart1}) with $u_2 =u_1$ and $\phi_2 =\phi_1$. Furthermore by letting $R \rightarrow \infty$, we obtain that 
 \begin{equation}\label{im0}
 \big(\Im (T) v_1, v_1 \big)_{ H^1(D  ) } =  \int\limits_{D} \Im(A) \nabla \phi_1 \cdot \nabla \overline{\phi_1}-k^2 \Im(n) |\phi_1|^2 \, dx- k \int\limits_{\mathbb{S} }|u^{\infty}_1|^2 \, ds(\hat{x}) 
 \end{equation}
where the far field pattern $u_1^\infty$ is defined by the asymptotic expansion in \eqref{ffpattern} and $ \mathbb{S}=\{ {x} \in \R^d\,  :  \, |{x}|=1 \}$ is the unit circle or sphere. 

{
We are going to prove coercivity by using a contradiction argument, therefore assume that $\exists \, v_\ell \in H^1(D)$ such that 
$$\| v_\ell \|_{H^1(D)}=1\quad  \text{where}  \quad \big|( \Im(T) v_\ell, v_\ell )_{ H^1(D) } \big| \longrightarrow 0 \quad \text{as} \quad \ell \rightarrow \infty.$$ 
This gives that $u_\ell$ defined by solving \eqref{sourcep1} with source $v=v_\ell$ is bounded in $H^1_{loc}(\R^m)$.
Using equation \eqref{im0}  as well as the fact that $\Im(A)$ is uniformly negative definite and $\Im(n) $ is uniformly positive in $D$ we conclude that $\phi_\ell \rightarrow 0$ in $H^1(D)$. By the well-posedness of \eqref{sourcepw} we have that 
$$\| u_\ell\|_{H^1(B_R)} \leq C \| \phi_\ell\|_{H^1(D)}  \longrightarrow 0 \quad \text{as} \quad \ell \rightarrow \infty.$$
Now since $v_\ell=\phi_\ell - u_\ell$ we have that $ v_\ell \rightarrow 0$ in $H^1(D)$ which gives that result by contradiction since $||v_\ell ||_{H^1(D)}=1$. 
 }
\end{proof}

{
\begin{corollary} 
\begin{enumerate}
\item Assume $\Im(A) =0$, $\Im(n) =0$ in $D$ and $A-I$ is positive (or negative) definite  uniformly in $D$ and $n-1$ is negative (or positive) uniformly in $D$. Then  $N_{\sharp}=  \big|\Re({N})\big| +\big|\Im({N})\big| $ is positive, compact, injective and has a dense range.
\item Assume that $\Im(A)$ is uniformly negative definite and $\Im(n) $ is uniformly positive in $D$. Then $N_{\sharp}= \sigma \Im({N})$ is positive, compact, injective and has a dense range where $\sigma =-1$.
\end{enumerate}  
\end{corollary}
 }

\subsubsection{Characterization of the scattering object } 
We now have all we need to prove the main results of this paper. 
{
Using Theorem \ref{fmthm2} along with Picard's criteria we will construct an indicator function to reconstruct the support of $D$ via the factorization method which depends on the analysis in the previous section. We also connect a modified linear sampling to the factorization method. The indicator function proposed for modified linear sampling here is motivated by the work done in \cite{GLSM}. It is shown in \cite{GLSM} that the `minimizer' $g_z^{\eps}$ of the functional 
$$ \mathcal{J}_{\eps} \big( \Phi( \cdot \, , z) \, ; \, g_z \big)  = \eps \big( N_{\sharp} \, g_z  , g_z  \big)_{L^2(C)}   +\|  Ng_z- \Phi( \cdot \, , z) \|^2_{L^2(C)}$$
(recall that the operator $N_{\sharp}$ is positive and injective) can be used to derive an indicator function given by 
$$\big( N_{\sharp} \, g_z^{\eps}  , g_z^{\eps}  \big)_{L^2(C)} = \big\| N_{\sharp}^{1/2} g_z^{\eps} \big\|^2_{L^2(C)} $$ 
which tends to infinity as the regularization parameter $\eps$ tends to zero for $z \notin D$ and is bounded otherwise. The square root of the operator is defined by the spectral decomposition where 
\begin{align}
N_{\sharp}^{1/2} g = \sum \lambda_j^{1/2} (g, \psi_j) \psi_j \label{squareroot}
\end{align}
for all $g \in L^2(C)$ where  $(\lambda_j , \psi_j) \in \R^+ \times L^2(C)$ is the  orthonormal eigensystem of $N_{\sharp}$. Here we extend this idea and show that any regularized solution to 
$$ N_{\sharp} \, g_z = \Phi( \cdot \, , z) $$
satisfies that $\big\| N_{\sharp}^{1/2} g_z^{\eps} \big\|_{L^2(C)}$ is bounded as the regularization parameter $\eps$ tends to zero if and only if $z \in D$.
}
In this section we will consider two cases, i.e. for an absorbing and non-absording obstacle. We will state the main results separately for each case for simplicity.

\begin{theorem} [Range test for a non-absorbing media]\label{maintheorem1}
Assume that $\Im(A)=0$ and $\Im (n)=0$ in $D$. {If $A-I$ is positive (or negative) definite  uniformly and $n-1$ is negative (or positive) uniformly in $D$} then for $z \in \Omega$ 
$$\Phi( \cdot \, , z) \in \text{Range} \big(N_\sharp^{1/2} \big)\quad  \text{ if and only if} \quad z \in D$$
where $N_\sharp = |\Re(N)| +|\Im(N)|$.
\end{theorem}
\begin{proof}
The result follows from the fact that Theorem \ref{fmthm2} along with Theorems \ref{rangeh} and \ref{T-analysis} gives that $\text{Range}(N_\sharp^{1/2} ) =\text{Range}(H^*)$ and then by appealing to Theorem \ref{rangeh1} proves the claim. 
\end{proof}

We now give a range test for an absorbing media. 

\begin{theorem} [Range test for an absorbing media]\label{maintheorem2}
Assume that {$\Im(A)$ is uniformly negative definite and $\Im(n) $ is uniformly positive in $D$} then we have that for $z \in \Omega$ 
$$\Phi( \cdot \, , z) \in \text{Range}\big(N_\sharp^{1/2} \big)\quad  \text{ if and only if} \quad z \in D$$
where {$N_\sharp = \sigma \Im({N})$ with $\sigma =-1$.}
\end{theorem}
\begin{proof}
{
Similarly we have that by using Theorem 4.1  in \cite{factor-music} for the imaginary part of the operator. Then the result follows from Theorems \ref{rangeh} and \ref{tkanalysis} gives that $\text{Range}(N_\sharp^{1/2} ) =\text{Range}(H^*)$ and then by appealing to Theorem \ref{rangeh1} proves the claim.  
}
\end{proof}

{
In either case we 
}
now let $\lambda_j  \in \R^+$ and $ \psi_j \in  L^2(C)$ be an orthonormal eigensystem of the {positive,} self-adjoint compact operator ${N}_{\sharp}$ then applying Picard's criterion (see for e.g. Theorem 2.7 of \cite{p1}) to Theorems \ref{maintheorem1} or \ref{maintheorem2} gives the following. 
{
\begin{corollary}\label{indicator1}
Assume that the assumptions of Theorem \ref{maintheorem1} or \ref{maintheorem2} are satisfied then for all $z \in \Omega$ 
$$\sum\limits_{j=1}^{\infty} \frac{\big| \big( \Phi(\cdot \, , z ) \,,\psi_j  \big)_{L^2(C)} \big|^2}{ \lambda_j }  < \infty \quad \text{if and only if } \quad z \in D .$$
\end{corollary}

Corollary \ref{indicator1} gives a way to characterize the scatterer from the spectral data of $N_{\sharp}$ and gives that 
$$ W(z)=\left[ \sum\limits_{j=1}^{\infty} \frac{\big| \big( \Phi(\cdot \, , z ) \,,\psi_j  \big)_{L^2(C)} \big|^2}{ \lambda_j } \right]^{-1} > 0 \quad \text{if and only if } \quad z \in D .$$
So one can reconstruct $D$ by plotting the function $W(z)$ since the support of the function coincided with the scatterer. 

We are now ready show that the modified linear sampling gives an equivalent indicator function for reconstructing $D$. To this end, we define the modified near field equation
\begin{align}
N_{\sharp} \, g_z = \Phi( \cdot \, , z)  \label{MNFE}
\end{align}
and the regularized solution to \eqref{MNFE} given by 
\begin{align}
g^{\eps}_z = \sum\limits_{j=1}^{\infty} \lambda_j f_\eps (\lambda_j^2) \,  \big( \Phi(\cdot \, , z ) \,,\psi_j  \big)_{L^2(C)} \,  \psi_j \label{reg-solu}
\end{align}
where $\lambda_j  \in \R^+$ and $ \psi_j \in  L^2(C)$ be an orthonormal eigensystem of the positive, self-adjoint compact operator ${N}_{\sharp}$. The filter $f_\eps (t) : \big(0 ,\lambda_1^2 \big] \longmapsto \R^+$ for the regularization scheme is a piece-wise continuous sequence of functions that satisfies for $0 < t \leq  \lambda_1^2$
$$\lim\limits_{\eps \to 0} f_\eps(t) =\frac{1}{t} \quad \text{ and } \quad  t f_\eps(t) \leq C_{\text{reg}} \quad \text{ for all } \eps > 0$$
where $\lambda_1$ is the largest eigenvalue of the operator $N_{\sharp}$. 
Some common filter-functions are given by 
\begin{align}
f_\eps (t) =  \frac{1}{t+\eps}, \quad \quad \displaystyle{ f(t)= \left\{\begin{array}{lr} \frac{1}{t} \, \, & \, \, t>\eps  \\
 				&  \\
 0\, \,  & \, \,   t \leq \eps
 \end{array} \right.}  \quad \text{and} \quad  f_\eps (t) =\frac{1-(1-at)^{1/\eps}}{t} \label{filters}
 \end{align}
where the constant $a< \lambda_1^2$. The filter functions in \eqref{filters} are for Tiknohov's regularization, Spectral cutoff and Landweber iteration,
 respectively. Notice that by \eqref{squareroot} and \eqref{reg-solu}
$$N_{\sharp}^{1/2} g_z^{\eps} = \sum\limits_{j=1}^{\infty} \lambda^{3/2}_j f_\eps (\lambda_j^2) \,  \big( \Phi(\cdot \, , z ) \,,\psi_j  \big)_{L^2(C)} \,  \psi_j$$ 
which implies that 
\begin{align*}
\big\| N_{\sharp}^{1/2} g_z^{\eps} \big\|^2_{L^2(C)} &= \sum\limits_{j=1}^{\infty} \lambda^{3}_j \big[f_\eps (\lambda_j^2)\big]^2 \,  \Big| \big( \Phi(\cdot \, , z ) \,,\psi_j  \big)_{L^2(C)} \Big| ^2  = \sum\limits_{j=1}^{\infty} \lambda^{4}_j  \big[ f_\eps (\lambda_j^2) \big]^2 \, \frac{ \Big|  \big( \Phi(\cdot \, , z ) \,,\psi_j  \big)_{L^2(C)} \Big| ^2 }{\lambda_j} \\  
									     & \leq  C^2_{\text{reg}} \sum\limits_{j=1}^{\infty} \, \frac{ \Big|  \big( \Phi(\cdot \, , z ) \,,\psi_j  \big)_{L^2(C)} \Big| ^2 }{\lambda_j}.
\end{align*}
By Corollary \ref{indicator1} we have that $\big\| N_{\sharp}^{1/2} g_z^{\eps} \big\|_{L^2(C)}$ is bounded as $\eps$ tends to zero for all $z \in D$. We now consider the case when $z \notin D$ and notice that 
\begin{align*}
\big\| N_{\sharp}^{1/2} g_z^{\eps} \big\|^2_{L^2(C)} &= \sum\limits_{j=1}^{\infty} \lambda^{3}_j \big[f_\eps (\lambda_j^2)\big]^2 \,  \Big| \big( \Phi(\cdot \, , z ) \,,\psi_j  \big)_{L^2(C)} \Big| ^2 \\
									     & \geq \sum\limits_{j=1}^{M} \lambda^{3}_j \big[f_\eps (\lambda_j^2) \big]^2 \, \Big| \big( \Phi(\cdot \, , z ) \,,\psi_j  \big)_{L^2(C)} \Big| ^2.   
\end{align*} 
where $M \in \N$ is finite. Now, letting $\eps \to 0$ then we conclude that 				  
\begin{align*}
\lim\inf \limits_{\eps \to 0} \big\| N_{\sharp}^{1/2} g_z^{\eps} \big\|^2_{L^2(C)} \geq   \sum\limits_{j=1}^{M} \, \frac{ \Big| \big( \Phi(\cdot \, , z ) \,,\psi_j  \big)_{L^2(C)} \Big| ^2 }{\lambda_j}
\end{align*}
which implies that $\big\| N_{\sharp}^{1/2} g_z^{\eps}  \big\|$ is unbounded as $\eps$ tends to zero since the right hand side of the inequality becomes unbounded as $M$ increases for $z \notin D$.  It is clear that if $C_{\text{reg}}=1$ which is the case for the filter functions given in \eqref{filters} (for $\eps$ small) then 
$$\sum\limits_{j=1}^{M} \, \frac{ \Big| \big( \Phi(\cdot \, , z ) \,,\psi_j  \big)_{L^2(C)} \Big| ^2 }{\lambda_j}  \leq \,  \lim\limits_{\eps \to 0} \big\| N_{\sharp}^{1/2} g_z^{\eps} \big\|^2_{L^2(C)} \,  \leq \sum\limits_{j=1}^{\infty} \frac{\Big| \big( \Phi(\cdot \, , z ) \,,\psi_j  \big)_{L^2(C)} \Big|^2}{ \lambda_j }. $$ 
This gives that in the limit the modified linear sampling method's indicator given by 
$$\lim\inf\limits_{\eps \to 0} \big\| N_{\sharp}^{1/2} g_z^{\eps} \big\|^2_{L^2(C)} = \lim\inf\limits_{\eps \to 0} \big( N_{\sharp} \, g_z^{\eps}  , g_z^{\eps}  \big)_{L^2(C)}$$ 
is equivalent to the factorization method. The analysis done here follows for any scatterer where either the near field or far field operator where the factorization method is valid. Here one does not need the corresponding near field or far field operator to have an eigenvalue decomposition as in \cite{arens} and \cite{arens-armin}. The above analysis gives that following results.

\begin{theorem} \label{MLSM}
Assume that the assumptions of Theorem \ref{maintheorem1} or \ref{maintheorem2} is satisfied then the regularized solution $g_z^\eps$ of \eqref{MNFE} for all $z \in \Omega$ satisfies that 
$$\lim\inf\limits_{\eps \to 0}\big\| N_{\sharp}^{1/2} g_z^{\eps} \big\|^2_{L^2(C)} = \lim\inf\limits_{\eps \to 0} \big( N_{\sharp} g_z^{\eps}  , g_z^{\eps}  \big)_{L^2(C)}  < \infty \quad \text{if and only if } \quad z \in D. $$
\end{theorem}
\begin{corollary}\label{MLSM2}
Assume that the assumptions of Theorem \ref{maintheorem1} or \ref{maintheorem2} is satisfied then the regularized solution $g_z^\eps$ of \eqref{MNFE} for all $z \in \Omega$ satisfies that 
$$\lim\inf \limits_{\eps \to 0} \big\| g_z^{\eps} \big\|_{L^2(C)} = \infty  \quad \text{ for all } \quad z \notin D. $$
\end{corollary}
Corollary \ref{MLSM2} gives that any regularized solution to \eqref{MNFE} becomes unbounded as the regularization parameter $\eps$ tends to zero. In practice the indicator function for the modified linear sampling method can be taken to be either 
$$ P(z) = \big\| N_{\sharp}^{1/2} g_z^{\eps} \big\|^{-1}_{L^2(C)} \quad \text{ or } \quad I(z)=\big\|  g_z^{\eps} \big\|^{-1}_{L^2(C)}$$
where the support of these functions can be used to reconstruct the scatterer $D$.  
}

Due to the rigorous range characterization given by the factorization method, we have the following uniqueness result for the inverse problem. 

\begin{corollary}[Uniqueness for the Inverse Problem]
Assume that  $D_1$ and $D_2$ are two penetrable anisotropic obstacles having material properties $A_1$, $n_1$  and $A_2$, $n_2$ that satisfy the assumptions for either Theorem \ref{maintheorem1} or \ref{maintheorem2} in the interior of $D_1$ and $D_2$, respectively. If the corresponding scattered fields are such that $u^s_1(x,y)=u^s_2(x,y)$ for all $x, \, y \in C$ for a fixed wave number $k$ (that satisfies assumption \ref{note}) then $D_1$ = $D_2$.
\end{corollary}

\noindent{\bf A numerical examples for unit ball in $\R^2$:} We will now give an explicit example of Theorem \ref{maintheorem1}
{
and \ref{MLSM}
where we assume that the coefficients are isotropic homogeneous in a disk of radius one. This is a simple example to numerically see the validity of the equivalence between the factorization method and the modified linear sampling method. This is not an extensive numerical study but we give this numerical evidence to illustrate the analysis in this section.
}
Therefore, the coefficient matrix is given by $A=a I$ where both $a$ and $n$ are constant. Now notice that the scattering problem \eqref{forwardprob} can be written as find $u\in H^1(B_1)$ and $u^s \in H^1_{loc}(\R^2 \setminus \overline{B}_1)$ such that 
\begin{eqnarray*}
a \Delta u +k^2 n u=0 \, \, \text{in } \, \, B_1\quad \text{and} \quad  \Delta u^s + k^2 u^s=0  \, \, \, &\textrm{ in }& \,\,  \R^2 \setminus \overline{B}_1 \\
u^s-u= - \overline{\Phi(x,y)} \quad  \textrm{ and }  \quad  \frac{\partial u^s}{\partial r} - a \frac{\partial u}{\partial r}= - \frac{\partial }{\partial r} \overline{\Phi(x,y)} \quad &\textrm{ on }& \partial B_1 ,
\end{eqnarray*}
where $u$ is the total field in $B_1$ and $u^s$ is the radiating scattered field in $\R^2 \setminus \overline{B}_1$. We make the ansatz that the solutions can be written as the following series 
$$u^s(x,y)=\sum\limits_{|m|=0}^{\infty} \alpha_m H^{(1)}_m( k |x|) \text{e}^{ \text{i} m\hat{x}} \quad \text{and} \quad u(x,y)=\sum\limits_{|m|=0}^{\infty} \beta_m \text{J}_m\left(k \sqrt{\frac{n}{a}} |x| \right)\text{e}^{\text{i}m\hat{x}}.$$
Here we adopt the notation that points in $\R^2$ are given in by the polar representation 
$$ x= |x| \big( \cos (\hat{x} ) , \sin(\hat{x}) \, \big)^{\top},$$ 
where $ \hat{x} \in [0 , 2\pi)$. 
Recall that for $|y| > |x|$ the fundamental solution has the expansion 
$$\Phi(x,y)=\frac{\text{i}}{4} \sum\limits_{|m|=0}^{\infty}  H^{(1)}_m( k |y|) \text{J}_m\left(k |x| \right) \text{e}^{\text{i}m(\hat{y}-\hat{x})}$$
(see for e.g. \cite{coltonkress}) where $\text{J}_m$ is the Bessel function and $H^{(1)}_m$ is the first kind Hankel function of order $m$. We will assume that $C=\partial B_2$ and therefore applying the boundary conditions we obtain that 

$$\begin{bmatrix}
  H^{(1)}_m( k ) & -\text{J}_m\left(k \sqrt{\frac{n}{a}}  \right) \\
  ( H^{(1)}_m(k))' & -\sqrt{n a}\text{J}{'}_m\left(k \sqrt{\frac{n}{a}}\right)
\end{bmatrix} 
\begin{bmatrix}
  \alpha_m \\
  \beta_m \end{bmatrix}= 
  \frac{\text{i}}{4} H^{(2)}_m( 2 k ) \text{e}^{-\text{i}m\hat{y}} 
\begin{bmatrix}
 \text{J}_m\left(k\right) \\
  \text{J}{'}_m (k ) \end{bmatrix}.
$$
Solving the linear system gives that 
$$\alpha_m =\frac{\text{i}}{4} H^{(2)}_m(2k) \text{e}^{-\text{i}m\hat{y}}  \frac{ \text{J}_m\left(k \sqrt{\frac{n}{a}}\right)\text{J}{'}_m(k)-\sqrt{n a}\text{J}{'}_m\left(k \sqrt{\frac{n}{a}}\right)\text{J}_m(k) }{ \text{J}_m\left(k \sqrt{\frac{n}{a}}\right) \left( H^{(1)}_m(k)\right)'-\sqrt{n a}\text{J}{'}_m\left(k \sqrt{\frac{n}{a}}\right)H^{(1)}_m(k) }$$
with $H^{(2)}_m$ being the second kind Hankel function of order $m$. Simple manipulations give that the scattered field for $x , y \in \partial B_2$ can be written as 
$$ u^s(x,y) = \frac{\text{i}}{4} \sum\limits_{|m|=0}^{\infty} \sigma_m \big| H^{(1)}_m(2k) \big|^2 \text{e}^{\text{i}m(\hat{x} -\hat{y})} ,$$
where the polar angles $ \hat{x} ,\hat{y} \in [0 , 2\pi)$ and we define 
$$ \sigma_m= \frac{ \text{J}_m\left(k \sqrt{\frac{n}{a}}\right)\text{J}{'}_m(k)-\sqrt{n a}\text{J}{'}_m\left(k \sqrt{\frac{n}{a}}\right)\text{J}_m(k) }{ \text{J}_m\left(k \sqrt{\frac{n}{a}}\right) \left( H^{(1)}_m(k)\right)'-\sqrt{n a}\text{J}{'}_m\left(k \sqrt{\frac{n}{a}}\right)H^{(1)}_m(k) }.$$

Given the above expression for the scattered field we have that the near field operator can be written as $N : L^2(0 , 2\pi) \longmapsto L^2(0 , 2\pi)$ 
$$(Ng)(\hat{x}) = \frac{\text{i}}{4} \int_0^{2 \pi} \sum\limits_{|m|=0}^{\infty} \sigma_m \big| H^{(1)}_m(2k) \big|^2 \text{e}^{\text{i}m(\hat{x} -\hat{y})}  \, g(\hat{y}) \, d\hat{y}.$$
In our examples we approximate the near field operator by truncating the series. Let $N_M$ be the truncated approximation of the near field operator for some $M \in \N$ given by 
$$(N_M g)(\hat{x}) = \frac{\text{i}}{4} \int_0^{2 \pi} \sum\limits_{|m|=0}^{M} \sigma_m \big| H^{(1)}_m(2k) \big|^2 \text{e}^{\text{i}m(\hat{x} -\hat{y})}  \, g(\hat{y}) \, d\hat{y}.$$
We apply Theorem \ref{maintheorem1} and \ref{MLSM} to the discretized version of the truncated near field operator ${\bf N}_M$ where a simple 64 point Riemann sum approximation is used to discretize the integral. 
Therefore we will plot the function 
{
$$W_M (z)=\left[ \sum\limits_{i=1}^{64} \frac{ \big|( \phi_ z \,,\psi^{M}_i)_{\ell^2} \big|^2}{ \lambda^M_i } \right]^{-1} \quad \text{for} \quad \phi_ z =  \Phi(\cdot \,  , z) $$
where $\lambda^M_i$ and $\psi^M_i$ are the eigenvalues and vectors of the matrix
$$\big( {\bf N}_M \big)_\sharp = \big|\Re\big({\bf N}_M\big)\big| +\big|\Im\big({\bf N}_M\big)\big|.$$ 
Now, let the vector $g_z^\eps$ be the solution to the discretized modified near field equation by a spectral cutoff where we use the numerical rank of the matrix as the cutoff parameter with indicator function 
$$P_M (z) = \Big| \Big( \big( {\bf N}_M \big)_\sharp\,  g_z^\eps , g_z^\eps \Big)_{\ell^2} \Big|^{-1}.$$
In the following examples we plot the functions $W_M (z)$ and $P_M (z)$ with $M=20$. 
}
 
\begin{figure}[H]
\hspace{-1in}\includegraphics[scale=0.45]{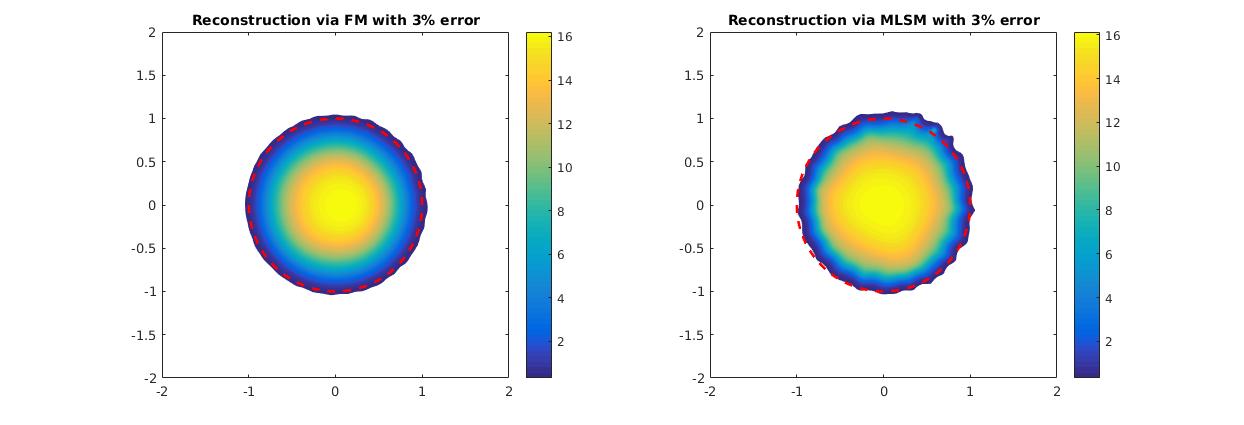}
\caption{ Reconstruction via the factorization and modified linear sampling methods where $A=(1/2)I$ and $n=5$. The dotted red line is the actual boundary of the unit circle.}
\end{figure}
 
\begin{figure}[H]
\hspace{-1in}\includegraphics[scale=0.45]{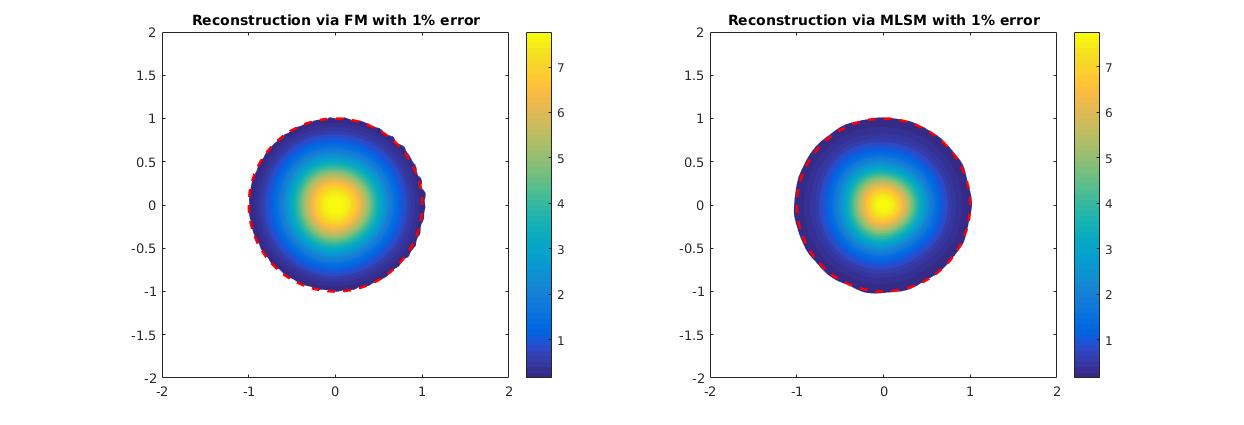}
\caption{Reconstruction via the factorization and modified linear sampling methods where $A=(3-\text{i})I$ and $n=1/4 +2 \text{i}$. The dotted line is the actual boundary of the unit circle.}
\end{figure}
  
\appendix
\section{Abstract theory of the factorization method}
We now recall the main theoretical results from \cite{armin} which we use to develop the range test in the previous section. To this end, let $X \subset U \subset X^*$ be a Gelfand triple with a Hilbert space $U$ and a reflexive Banach space $X$ such that the embedding is dense. Furthermore, let $Y$ be a second Hilbert space and let ${N}: Y \mapsto Y$, $H: Y \mapsto X$ and $T: X \mapsto X^*$ be linear bounded operators such that $N=H^*TH$.
\begin{theorem}(Theorem 2.1 in \cite{armin}) \label{fmthm2}
Assume that 
\begin{enumerate}
\item $H$ is compact and injective.
\item  $\Re(T)$ is the sum of a compact operator and a self adjoint coercive operator. 
\item $\Im(T)$ is non-negative(or non-positive) on $X$.\\
\hspace*{-1.2cm} Moreover assume that either of the following is satisfied:
\item $T$ is injective. 
\item $\Im(T)$ is strictly positive(or negative) on the null space of $\Re (T)$.
\end{enumerate}
Then the operator $N_{\sharp}= |\Re(\text{e}^{\text{i}t} N)| +|\Im(N)|$ is positive, and the range of the operators $H^*: X^* \mapsto Y$ and $N_{\sharp}^{1/2} : Y \mapsto Y$ coincide. 
\end{theorem}


\end{document}